\documentclass[10pt,twoside]{siamltex}
\usepackage{amsfonts,epsfig}
\usepackage{lineno,hyperref}

\usepackage{graphicx}
\usepackage{amsmath}
\usepackage{amssymb}

\usepackage{color}
\newcommand{\tc}{\textcolor{black}}
\usepackage{eqparbox}
\usepackage[section]{algorithm}
\usepackage{algorithmic}

\newcommand{\argmin}{\operatorname*{arg\,min}}

\newcommand{\R}{\mathbb R}

\newcommand{\colspan}{\operatorname{ran}}
\newcommand{\sign}{\operatorname{sign}}
\newcommand{\U}{\mathcal{U}}

\makeatletter
\renewcommand*\env@matrix[1][*\c@MaxMatrixCols c]{%
  \hskip -\arraycolsep
  \let\@ifnextchar\new@ifnextchar
  \array{#1}}
\makeatother

\title{A geometric note on subspace updates and orthogonal matrix decompositions under rank-one modifications}

\author{
  Ralf Zimmermann\thanks{Department of Mathematics and Computer Science, University of Southern Denmark (SDU) Odense,
    (zimmermann@imada.sdu.dk).}
}
\begin{document}
\maketitle

\begin{abstract}
 In this work, we consider rank-one adaptations $X_{new} = X+ab^T$ of a given matrix $X\in \mathbb{R}^{n\times p}$
 with known matrix factorization $X = UW$, where $U\in\mathbb{R}^{n\times p}$ is column-orthogonal, i.e. $U^TU=I$.
 Arguably the most important methods that produce such factorizations are the singular value decomposition (SVD),
 where $X=UW=U\Sigma V^T$, and the QR-decomposition, where $X = UW = QR$.
 An elementary approach to produce a column-orthogonal matrix $U_{new}$,
 whose columns span the same subspace as the columns of the rank-one modified $X_{new} = X +ab^T$
 is via applying a suitable coordinate change such that in the new coordinates, the update affects a single column and 
 subsequently performing a Gram-Schmidt step for reorthogonalization. This may be interpreted as a rank-one adaptation of the $U$-factor in the SVD or a rank-one adaptation of the $Q$-factor in the QR-decomposition, respectively, and leads to 
 a decomposition for the adapted matrix $X_{new} = U_{new}W_{new}$.
 By using a geometric approach, we show that this operation is equivalent to traveling from the subspace $\mathcal{S}= \colspan(X)$
 to the subspace $\mathcal{S}_{new} =\colspan(X_{new})$ on a geodesic line on the Grassmann manifold and we
 derive a closed-form expression for this geodesic.
 In addition, this allows us to determine the subspace distance between
 the subspaces $\mathcal{S}$ and $\mathcal{S}_{new}$ without additional computational effort.
 Both $U_{new}$ and $W_{new}$ are obtained via elementary rank-one matrix updates
 in $\mathcal{O}(np)$ time for $n\gg p$.
 
 Possible fields of applications include subspace estimation in computer vision, signal processing, update problems in data science and adaptive model reduction.
\end{abstract}


\section{Introduction}
Investigations on the behavior of matrix decompositions under perturbations of restricted rank have a long tradition
\cite{balzano2010online, Brand2006,BunchNielsenSorensen1978,Kaufman1976,GillGolubMurraySaunders1974, Gu:1994:SEA:196045.196076,   MoonenVanDoorenVandewalle1992, Thompson1976}.
Of special importance in many applications are rank-one modifications $X_{new} = X +ab^T$ of a given matrix $X\in \R^{n\times p}$ with either known (thin) SVD $X = U \Sigma V^T$, where $U\in \R^{n\times p}$, $\Sigma, V \in \R^{p\times p}$,
or known (compact) QR-decomposition $X = QR$, where $Q\in \R^{n\times p}$, $R \in \R^{p\times p}$.
In both cases, the matrix decomposition is of the form $X=UW$ and
the columns of $U$ and $Q$, respectively, provide an {\em orthonormal basis} for the range of $X$, i.e., the subspace $\colspan(X)$. For an introduction to subspace computations and updating matrix factorizations as well as additional references, the reader may consult \cite[\S 6.4, 6.5]{GolubVanLoan4th}.

We work in the setting, where $a\not\in \colspan{X}$. This guarantees that the modified subspace $\colspan(X_{new})$ satisfies 
\[
  \colspan(X_{new}) \neq \colspan(X),\quad 
 \dim(\colspan(X_{new})) = \dim(\colspan(X)).
\]
We refer to this as a {\em non-trivial rank-preserving subspace modification.}
The main original contribution of this note is a proof that an orthogonal matrix factor $U_{new}$,
i.e., an orthonormal basis of the updated column-span can be reached via a geodesic path on the Grassmann manifold \cite{AbsilMahonySepulchre2008}, that starts in $U$ (resp. $Q$) with a suitable tangent velocity $\Delta\in\R^{n\times p}$, where $\Delta$ is a rank-one matrix from the tangent space of the Grassmann manifold.
This establishes a closed-form expression for $U_{new}$, which, in turn, leads to a closed-form expression for the updated matrix factorization $X_{new} = X+ab^T = U_{new}W_{new}$.
It turns out that both $U_{new}$ and $W_{new}$ are obtained via a standard rank-one update on 
$U$ and $W$, respectively.
One may consider this as a formula for updating the orthonormal basis of a given subspace under arbitrary, non-degenerative rank-one modifications
of the subspace.

Possible applications present themselves in subspace tracking \cite{MoonenVanDoorenVandewalle1992},
computer vision \cite{Chandrasekaran1997, HallMarshallMartin2000} and adaptive model reduction \cite{pehersto15lowrank,ZimmermannWillcox2016, ZimmermannPeherstorferWillcox_2017}.
\subsection{Main result: The update formula}
Let $X =  UW\in \R^{n\times p}$ with $U\in \R^{n\times p}, W\in \R^{p\times p}$ such that $U^TU = I$.
Let $a\in \R^{n}\setminus \colspan(X), b\in\R^p$
and consider the rank-one update $X_{new} = X + ab^T = U W+ ab^T$.
Then, the updated decomposition can be written as 
\begin{align*}
 X_{new} &= U_{new} W_{new}, \mbox{ with } & U_{new} &= U + \left(\alpha Uw + \beta q\right) w^T,\\
 && W_{new} &= W + \left(U^Ta + \gamma w\right) b^T.
\end{align*}
Here, $q = \frac{ (I-UU^T)a}{\|(I-UU^T)a\|_2}\in\R^{n}$,  $w = \frac{-W^{-T}b}{\|W^{-T}b\|_2}\in\R^{p}$.
The coefficients $\alpha, \beta, \gamma\in\R$ and further details are specified in Theorem \ref{thm:main} 
in Section \ref{sec:main}. 
The directions $q$ and $w$ and the coefficients are such that $\colspan(X_{new}) = \colspan(U_{new})$ and 
$U_{new}$ has orthogonal columns, i.e., $U_{new}^TU_{new} = I$.
Note that both the update on $U$ and the update on $W$ can be conducted via the Level-2 BLAS function ``DGER''
which performs the operation $A := \alpha xy^T + A$.
\subsection{The case \texorpdfstring{$a\in$ ran$(X)$}{a in ran(X)}}
If $a\in \colspan(X)$, then there exists a coefficient vector $x\in \R^p$ such that $a = Xx$. As a consequence,
$\colspan(X + ab^T) = \colspan(X(I_p + x b^T)) \subset \colspan(X)$. If $I_p + xb^T$ has full rank, then the subspace is {\em preserved}.
If $I_p + xb^T$ is rank-deficient, then the subspace {\em deflated}. Obviously, this happens,
if there is $\tilde{b}$ such that $b = -\frac{1}{\tilde{b}^Tx}\tilde{b}$, in which case $x\in \ker(I_p + xb^T)$.
In fact, by the Wedderburn rank reduction theorem, this already fully characterizes the case of deflation,  see Section \ref{sec:related}.
Moreover, the Wedderburn theory shows that for the rank to decrease, necessarily $a\in \colspan(X)$.
Therefore, we restrict our considerations to the case $a\not\in \colspan(X)$.
\subsection{Notation and preliminaries}
\label{sec:NotAndPrelims}
The $(p\times p)$-{\em identity matrix} is denoted by $I_{p}\in\R^{p\times p}$, or simply by $I$, if the dimensions are clear.
The \tc{$(p\times p)-$}{\em orthogonal group}, i.e., the set of all
square orthogonal matrices, is denoted by
\[
  O(p) = \{R \in \R^{p\times p}| R^TR = RR^T = I_{p}\}.
\]
For a matrix $X\in \R^{n\times p}$, the subspace spanned by the columns of $X$ is called the {\em range} of $X$ and is
denoted by $\mathcal{X} := \colspan(X):= \{X\alpha\in \R^n| \quad \alpha \in \R^p\}\subset \R^n$.
We also speak of {\em the subspace spanned by} $X$.
The set of all $p$-dimensional subspaces $\mathcal{X}\subset \R^n$
forms the {\em Grassmann manifold}
\[
  Gr(n, p):= \{\mathcal{X}\subset \R^n| \quad \mathcal{X} \mbox{ subspace, dim}(\mathcal{X}) = p\}.
\]

The {\em Stiefel manifold} is the compact matrix manifold of all column-orthogonal
rectangular matrices
\[
  St(n, p):= \{U \in \R^{n\times p}| \quad U^TU = I_{p}\}.
\]
The Grassmann manifold can be realized as a quotient manifold
of the Stiefel manifold
\begin{equation}
\label{eq:Grassmann_quotient}
  Gr(n, p) = St(n, p)/O(p) = \{[U]| \quad U\in St(n, p)\},
\end{equation}
where $[U] = \{UR| \quad R\in O(p)\}$ is the {\em orbit}, or {\em equivalence class}
of $U$ under actions of the orthogonal group.
Hence, by definition, two matrices $U,\tilde{U} \in St(n, p)$ are in the same $O(p)$-orbit
if they differ by a $(p\times p)$-orthogonal matrix:
\[
  [U] = [\tilde{U}] :\Leftrightarrow \exists R\in O(p): U = \tilde{U}R.
\]
A matrix $U\in St(n, p)$ is called a {\em matrix representative} of a subspace $\mathcal{U}\in Gr(n, p)$,
if $\mathcal{U}=\colspan(U)$.
We will also consider the orbit $[U]$ and the subspace $\mathcal{U}=\colspan(U)$ as the same object.

The {\em tangent space} $T_{[U]}Gr(n, p)$ at a point $[U] \in Gr(n, p)$
can be thought of as the space of velocity vectors of differentiable curves on $Gr(n, p)$
passing through $[U]$.
For any matrix representative $U\in St(n, p)$ of $[U]\in Gr(n, p)$,
the tangent space of $Gr(n, p)$ at $[U]$ is represented by
\begin{equation}
\label{eq:GrassTangent}
 T_{[U]}Gr(n, p) = \left\{\Delta \in \R^{n\times p}|\quad \Delta^TU = 0\right\}\subset \R^{n\times p},
\end{equation}
its canonical metric being $\langle \Delta, \tilde{\Delta}\rangle_{Gr} = tr(\Delta^T\tilde{\Delta})$,
\cite[\S 2.5]{EdelmanAriasSmith1999}.
Endowing each tangent space with this metric turns $Gr(n, p)$ into a {\em Riemannian manifold}.
As in \cite{EdelmanAriasSmith1999}, we will make use throughout of the quotient representation (\ref{eq:Grassmann_quotient})
of the Grassmann manifold with matrices in $St(n, p)$ acting as representatives in numerical computations.

Of special importance to this work are the {\em geodesic lines} on the Grassmann manifold.
Geodesics on curved manifolds can be considered as the generalization of straight lines in flat, Euclidean spaces.
From general differential geometry \cite{DoCarmo2013riemannian}, it is known that a geodesic $t\mapsto [U](t)$ is specified by a
second-order differential equation and is thus uniquely determined by a starting point $[U] = [U](0)$ and a starting velocity $\Delta = [\dot U](0) \in T_{[U]}Gr(n, p)$. This unique dependency gives rise to the so-called Riemannian exponential function 
\[
  Exp_{[U]}: T_{[U]}Gr(n, p) \rightarrow Gr(n, p), \quad
  \Delta \mapsto Exp_{[U]}(\Delta).
\]
The associated geodesic is $t\mapsto [U](t) = Exp_{[U]}(t\Delta)$, see Fig. \ref{fig:manifold_plot_SVD_up2} for an illustration.
An explicit formula for Riemannian exponential and thus for the geodesics on the Grassmannian was derived in 
\cite[\S 2.5.1]{EdelmanAriasSmith1999}. For a given pair of initial values $U\in Gr(n,p)$, $\Delta\in T_{[U]}Gr(n,p)$, the corresponding geodesic is
\begin{equation}
 \label{eq:grassgeo}
 t \mapsto \left[ U\Psi\cos(t S)\Psi^T + \Phi\sin(t S)\Psi^T \right] \in Gr(n,p),
 \quad \Phi S\Psi^T\stackrel{\mbox{SVD}}{=}\Delta,
\end{equation}
where $\Phi\in St(n,p)$, $\Psi\in O(p)$ and $S\in \R^{p\times p}$ diagonal.\footnote{It is understood that $\cos$ and $\sin$ act only on the diagonal elements of $tS$ in eq. \eqref{eq:grassgeo}.}

For a rectangular, full column-rank matrix $X\in \R^{n\times p}$, the {\em orthogonal projection}
onto the column span of $X$ is
\begin{equation}
\label{eq:Def_OrthoProj}
  \Pi_X: \R^n \rightarrow \colspan(X), \quad x\mapsto (X(X^TX)^{-1}X^T)x.
\end{equation}
An orthonormal basis (ONB) $\{u^1,\ldots,u^p\}\subset \R^n$ of $\colspan(X)$ gives rise to a
matrix $U=(u^1,\ldots,u^p)\in St(n,p)$ and the orthogonal projection reduces to
$ \Pi_X: x\mapsto UU^Tx$.

The {\em principal angles} (aka canonical angles) $\theta_1,\ldots,\theta_p\in[0,\frac{\pi}{2}]$ between two subspaces $[U], [\tilde{U}]\in Gr(n, p)$
are defined recursively by
\[
\cos(\theta_k) := u_k^Tv_k := 
\max_{\begin{array}{l} u\in [U], \|u\|=1\\
       u\bot u_1,\ldots, u_{k-1}
      \end{array}
     }
  \max _{\begin{array}{l} v\in [\tilde{U}], \|v\|=1\\
          v\bot v_1,\ldots, v_{k-1}
         \end{array}
        }
 u^Tv.
\]
The principal angles can be computed via
$\theta_k := \arccos(\sigma_k) \in [0,\frac{\pi}{2}]$,
where $\sigma_k$ is the $k$-largest singular value of $U^T\tilde{U}\in\R^{p\times p}$
\cite[\S 6.4.3]{GolubVanLoan4th}.
The {\em Riemannian subspace distance} between $[U], [\tilde{U}]\in Gr(n, p)$ is
\begin{equation}
 \label{eq:SubspaceDist}
 \mbox{dist}([U], [\tilde{U}]) := \|\Theta\|, \quad \Theta = (\theta_1,\ldots,\theta_p)\in\R^p,
\end{equation}
see \cite[\S 2.5.1, \S 4.3]{EdelmanAriasSmith1999}.
Here and throughout, $\|\cdot\|$ denotes the Euclidean norm.
\section{Problem statement and review of the state-of-the-art}
\label{sec:probstate}
Let $a\in \R^{n}$, $b\in \R^{p}$ and consider the rank-one update
\[
  X_{new} = X + ab^T \in \R^{n \times p}.
\]
Suppose that $X$ has full column-rank $p$. Let $X= U\Sigma V^T$ denote the (thin) SVD of $X$, 
where $U\in St(n,p)$, $V\in O(p)$ and $\Sigma$ is a {\em regular} $p$-by-$p$ diagonal matrix.
Let 
$$
X_{new} = X + ab^T = U\Sigma V^T + ab^T = U_{new}\Sigma_{new} V_{new}^T
$$ denote the updated (thin) SVD after the rank-one modification.
By writing 
\[
 X_{new} = U\Sigma V^T + ab^T = \left(U + ab^TV\Sigma^{-1}\right) \Sigma V^T,
\]
we see that $\colspan(U_{new}) = \colspan(X_{new}) = \colspan(U + ab^TV\Sigma^{-1} )$. Hence, the rank-one update on $X$ acts as a rank-one update on $U$,
which can be considered as an ONB matrix representative $U\in St(n,p)$ for the subspace $\colspan(X)$.
\paragraph{Objective} The task is to find an {\em orthogonal} subspace representative $\tilde{U}_{new}\in St(n,p)$ such that 
$[\tilde{U}_{new}] = [U_{new}]$, i.e.,
such that $\tilde{U}_{new}$ spans the same subspace as the updated $U_{new}$.
\subsection*{Review: rank-one adaptations}
\label{sec:ROreview}
The standard way to approach the above objective is via rank-one SVD updates as considered in
\cite{BunchNielsen1978, Brand2006} and is briefly reviewed below.

The scheme of \cite{Brand2006} starts as follows:
The rank-one update $X + ab^T = U\Sigma V^T + ab^T$ is written in factorized form as
\begin{subequations}
 \begin{align}
 X + ab^T &= \left(U, a\right) 
 \begin{pmatrix}[c|c]
  \Sigma & 0\\
  \hline 
  0   & 1
 \end{pmatrix}
  \begin{pmatrix}[c]
  V^T\\
  b^T
 \end{pmatrix}\\
 &= \left(U, q\right) 
 \begin{pmatrix}[c|c]
  I_p & U^Ta\\
  \hline
  0   & \|\tilde{q}\|
 \end{pmatrix}
 \begin{pmatrix}[c|c]
  \Sigma & 0\\
  \hline 
  0   & 1
 \end{pmatrix}
  \begin{pmatrix}[c|c]
  I_p & 0\\
  \hline 
  b^TV & 0
 \end{pmatrix}
 \begin{pmatrix}[c]
  V^T \\
  \hline 
  0
 \end{pmatrix}
 \\
 &= \left(U, q\right) 
 \begin{pmatrix}[c|c]
  I_p & U^Ta\\
  \hline
  0   & \|\tilde{q}\|
 \end{pmatrix}
 \begin{pmatrix}[c|c]
  \Sigma & 0\\
  \hline 
  0   & 1
 \end{pmatrix}
  \begin{pmatrix}[c]
  I_p \\
  \hline 
  b^TV
 \end{pmatrix}
 V^T\\
 \label{eq:brandupdate1}
 &= \left(U, q\right)
 \biggl(
 \underbrace{
 \begin{pmatrix}
  \Sigma\\
  \hline
  0
 \end{pmatrix}
 +
  \begin{pmatrix}
  U^Ta\\
  \hline
  \|\tilde{q}\|
 \end{pmatrix}
 b^TV
 }_{=:K\in \R^{(p+1)\times p}}
 \biggr) V^T,
 \end{align}
\end{subequations}
where $\tilde{q} = (I-UU^T)a\neq 0$ is the orthogonal component of $a$ with respect to the subspace $[U]$
and $q = \frac{\tilde{q}}{\|\tilde{q}\|}$.

Note that the left and rightmost matrix factors in the decomposition \eqref{eq:brandupdate1} are (column-)orthogonal by construction.
Hence, the updated SVD is obtained by computing the SVD of the (usually small) matrix $K = U'\Sigma'V'^T\in\R^{(p+1)\times p}$
and setting $U_{new} =  \left(U, q\right) U'$ as well as $V_{new} = VV'$.
It is worth noting that the procedure works by taking a detour via a representative $(U, q)\in St(n,p+1)$
of a $(p+1)$-dimensional subspace
that is pulled back to $St(n,p)$ by the factor  $U'\in St(p+1,p)$ from the SVD of $K$.  
Mind also that an auxiliary numerical computation of the SVD of the $(p+1)$-by-$p$-matrix $K$ is required.
%

In this work, the focus is on the updated subspace $[U_{new}] \in Gr(n,p)$.
Hence, for obtaining the subspace $[U_{new}]$, a QR-decomposition of $K$ can be conducted as an alternative to the SVD of $K$.
Rank-one adaptations of the QR-decomposition are investigated in \cite{Kaufman1976}.
The idea is analogous to the one outlined above:
The procedure of \cite[p. 775]{Kaufman1976} also starts with a detour via $p+1$ columns:
Let $X = QR$ with $Q\in St(n,p)$:
\begin{equation}
 \label{eq:QR_rank1}
  X + ab^T = \left(Q, a\right) 
 \begin{pmatrix}[c]
  R\\
  \hline 
  b^T
  \end{pmatrix}
  =
  \left(Q, q\right)
   \biggl(
   \underbrace{
   \begin{pmatrix}[c]
  R\\
  \hline 
  0^T
  \end{pmatrix}
  +
     \begin{pmatrix}[c]
  Q^Ta\\
  \hline 
  \|\tilde{q}\|
  \end{pmatrix}
  b^T}_{:= \tilde{K}\in \R^{(p+1)\times p}}
  \biggr),
\end{equation}
where, as before, $\tilde{q} = (I-QQ^T)a$, $q = \frac{q}{\|\tilde{q}\|}$. 
The algorithm proceeds with applying a suitable sequence of Givens rotations 
to reestablish the QR-decomposition.

In regards of the work at hand, it is important to emphasize that both the SVD-based approach of \cite{Brand2006} and the QR-based approach of \cite{Kaufman1976} make the same detour via $p+1$ columns and that both require an algorithmic decomposition of the auxiliary matrices $K$ and $\tilde{K}$, respectively.
To the best of the author's knowledge, there is no closed-form solution to these subproblems. The approach of this work eventually avoids the $(p+1)$-columns detour and produces a closed formula for the {\em updated subspace}.
%
%
\section{An elementary approach}
\label{sec:simple}
For simplicity, consider the update problem $X_{new}= U +ab^T$, where $U\in St(n,p)$ is already column-orthogonal.
Then an orthogonal subspace representative $U_{new}$ with $\colspan(U_{new}) = \colspan(X_{new})$ may be obtained in the following
straightforward manner:
Introduce the orthogonal coordinate change 
$Z = \begin{pmatrix}
  \frac{b}{\|b\|} & \hat{Z}
\end{pmatrix}
$,
where $\hat{Z}$ is an orthogonal completion such that $Z\in O(p)$.
Then $(U +ab^T)Z = UZ + \|b\| a e_1^T$. Writing $\tilde{U} = UZ$, in the new coordinates, the update is just an update on
the first column $(U +ab^T)Z = \left(\tilde{u}^1 + \|b\| a, \tilde{u}^2,\ldots, \tilde{u}^p \right)$.
To obtain a valid $U_{new}$, one needs only to reorthogonalize the first column $\tilde{u}^1 + \|b\| a$ against
the remaining columns, i.e., replace $v = \tilde{u}^1 + \|b\| a$ with 
$\tilde{u}^*:=\frac{v - \left(\tilde{u}^2,\ldots, \tilde{u}^p \right)\left(\tilde{u}^2,\ldots, \tilde{u}^p \right)^Tv}{\|v - \left(\tilde{u}^2,\ldots, \tilde{u}^p \right)\left(\tilde{u}^2,\ldots, \tilde{u}^p \right)^Tv   \|}$.\\
Because of $\left(\tilde{u}^2,\ldots, \tilde{u}^p \right)\left(\tilde{u}^2,\ldots, \tilde{u}^p \right)^T = UU^T - \tilde{u}^1(\tilde{u}^1)^T$
and $\tilde{u}^1 = U\frac{b}{\|b\|}$, an orthogonal subspace representative $\tilde{U}_{new}$ is readily found to be
$$
\tilde{U}_{new} = \left(\tilde{u}^*,\tilde{u}^2,\ldots, \tilde{u}^p \right), 
\quad \tilde{u}^*= \frac{(1 + a^TUb)\tilde{u}^1 + \|b\|\tilde{q}}{\|(1 + a^TUb)\tilde{u}^1 + \|b\|\tilde{q}\|},$$
where, as before, $\tilde{q} =  (I-UU^T)a$. Since $\tilde{u}^1 \bot \tilde{q}$, the denominator in the expression of
$\tilde{u}^*$ equals $\|(1 + a^TUb)\tilde{u}^1 + \|b\|\tilde{q}\|= \|\tilde{q}\|\|g\|$ with $\|g\|:= \left(\frac{(1 + a^TUb)^2}{\|\tilde{q}\|^2}+ \|b\|^2\right)^{1/2}$. (The underlying vector $g$ will be introduced in Theorem \ref{thm:main}.)
Thus, we obtain
\[
 \tilde{U}_{new} = UZ + \left(\frac{1 + a^TUb}{\|\tilde{q}\|} -1\right)\tilde{u}^1e_1^T + \frac{\|b\|}{\|g\|} q e_1^T.
\]
Transforming back to the original coordinates gives
\begin{equation}
\label{eq:simple}
  U_{new}:= \tilde{U}Z^T = U + \left(\left(\frac{1 + a^TUb}{\|\tilde{q}\|\|g\|} -1\right)\tilde{u}^1 + \frac{\|b\|}{\|g\|} q\right)\frac{b^T}{\|b\|}=:U + vw^T.
\end{equation}
From this expression, it is clear, that the matrix $Z$ need not be formed explicitly.
The original update $U+ab^T$ is replaced with $U+vw^T$, which produces a column-orthogonal matrix.
Moreover, since, $\colspan(U_{new}) = \colspan(U + ab^T)$, the update can be written in factorized form as
\begin{equation}
 U + ab^T =U_{new}U_{new}^T( U + ab^T) = U_{new}\underbrace{(I +  \tilde{v}w^T)}_{:=\hat{K}\in \R^{p\times p}}.
\end{equation}
This form compares to \eqref{eq:brandupdate1}, \eqref{eq:QR_rank1}
but avoids the detour via going to $p+1$ columns.
Reestablishing the SVD or the QR can be achieved by decomposing $\hat{K}= (I +  \tilde{v}w^T)$ accordingly.
In the next section, it is shown that the above procedure actually corresponds to traveling
along a geodesic line from $[U]$ to $[U_{new}]$ on the Grassmannian.
\section{General geometric rank-one subspace adaptation}
\label{sec:main}
%
\begin{figure}[ht]
\centering
\includegraphics[width=0.75\textwidth]{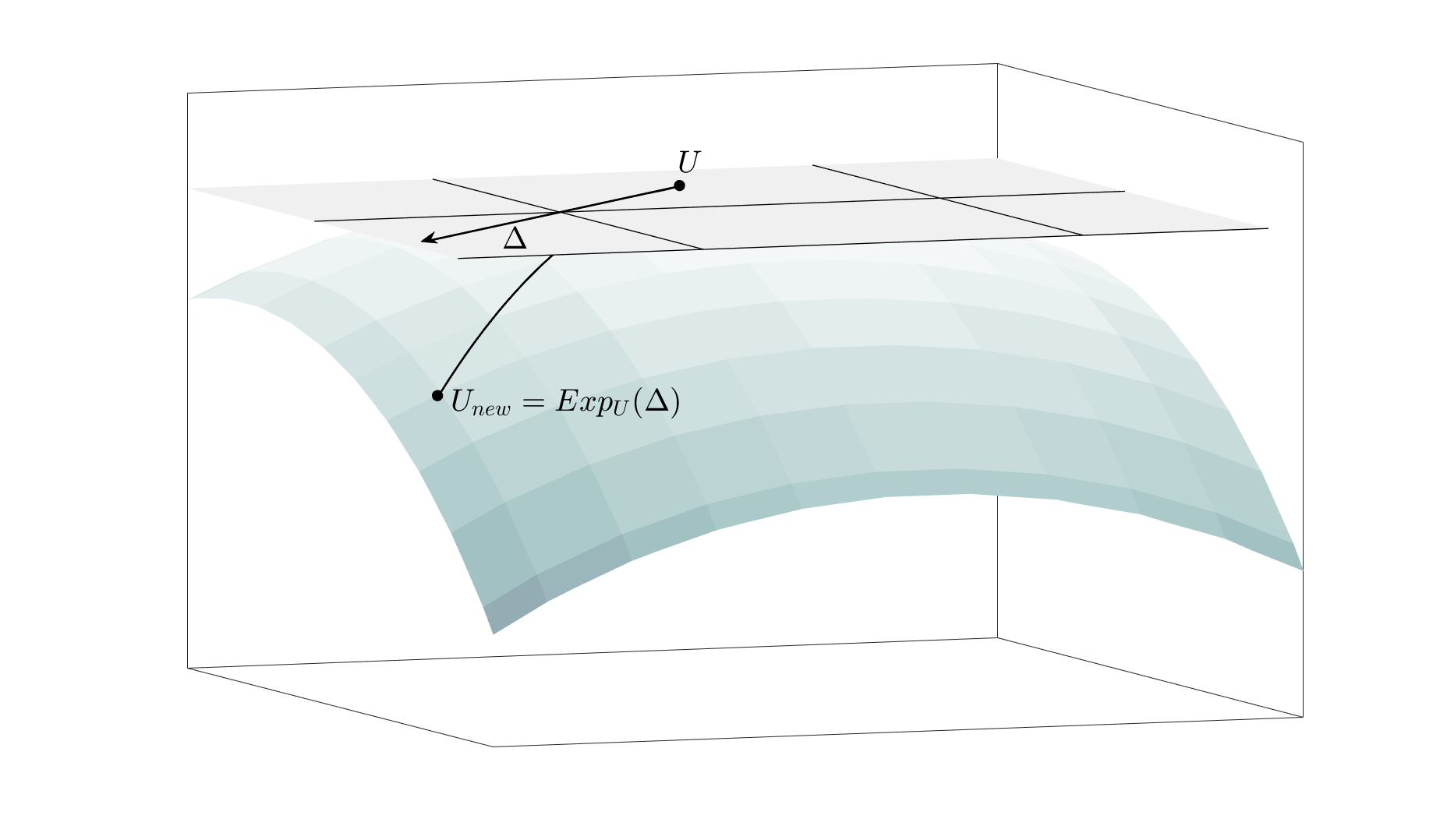}
\caption{Visualization of the Grassmann manifold $Gr(n,p)$ (curved surface)
with the tangent space $T_UGr(n,p)$ in $U$ (flat plane).
The vector $\Delta$ in the tangent space acts as the starting velocity vector of
the geodesic (solid line) starting in $U\in Gr(n,p)$.
}
\label{fig:manifold_plot_SVD_up}
\end{figure}
%
In this section, we present a geometric approach to the rank-one subspace adaptation problem introduced in Section \ref{sec:probstate}.

First, we formalize the notion of rank-one subspace adaptations. Just as the notion of a subspace itself, this concept should not depend on the matrix representatives.
\begin{definition}
\label{def:ROmods}
 Let $\U, \tilde{\U}\in Gr(n,p)$ be two subspaces of dimension $p$ in $\R^n$.
 We say that $\U$ and $\tilde{\U}$ differ by a non-trivial {\em rank-one modification},
 if all but one of the principle angles between $\U$ and $\tilde{\U}$ are zero., i.e., if
 $\theta_1=\cdots=\theta_{p-1} = 0$, $\theta_p >0$, 
\end{definition}

On the level of matrix subspace representatives, this corresponds indeed to rank-one matrix modifications, as the next lemma shows.
\begin{lemma}
\label{lem:ROchar}
 Two subspaces $\U,\tilde{\U}\in Gr(n,p)$ differ by a non-trivial rank-one modification if and only if
 there exist subspace representatives $U,\tilde{U}\in St(n,p)$ with $[U]=\U, [\tilde{U}]=\tilde{\U}$ and non-zero vectors $x\in \R^n\setminus\colspan(U), y\in \R^p$ such that 
 $\tilde{U} = U + xy^T.$
\end{lemma}
\begin{proof}
 `$\Rightarrow$' Let $U,\tilde{U}\in St(n,p)$ be arbitrary matrix representatives of the subspaces $\U$ and $\tilde{\U}$, respectively.
 Then, $U^T\tilde{U} \stackrel{\mbox{SVD}}{=} R S \tilde{R}^T$,  where $R,\tilde{R} \in O(p)$
 and $S = \mbox{diag}(1,\ldots,1, s_p)$, $0\leq s_p<1$, so that $\theta_p = \arccos(s_p)>0$ is the only non-zero principal angle.
 W.l.o.g., we replace the representatives $U,\tilde{U}$ with $UR,\tilde{U}\tilde{R}$,
 so that after this coordinate change, we have
 \[
  U^T \tilde{U} = \left(\langle u^j, \tilde{u}^k \rangle \right)_{j,k=1,\ldots,p}
  = S = 
  \begin{pmatrix}[c|c]
    I_{p-1} &   \\
    \hline
        & s_p \\
  \end{pmatrix}.
 \]
 In particular, by Cauchy-Schwartz' inequality,  $1 = \langle u^k, \tilde{u}^k \rangle \leq \|u^k\|\|\tilde{u}^k\| = 1$, $k=1,\ldots,p-1$ and, as a consequence $u^k =  \tilde{u}^k$ for $k=1,\ldots,p-1$.
 In summary,   $[U]=\U, [\tilde{U}] = \tilde{\U}$ and
 \[
  \tilde{U} = U + xy^T, \mbox{ with } x = (\tilde{u}^p - u^p),\quad y^T = e_p^T= (0,\ldots,0,1).
 \]
 Observe that $U^Tx = (0,\ldots,0,s_p -1)^T\neq 0$ and that $x\not\in \colspan(U)$. Otherwise, $\tilde{u}^p$ and $u^p$ would be collinear in contradiction to the fact that $s_p = \langle u^p, \tilde{u}^p \rangle < 1$. \\
`$\Leftarrow$' Suppose that $x\in \R^n\setminus\colspan(U), y\in \R^p\setminus \{0\}$ are such that $\tilde{U} = U + xy^T\in St(n,p)$.
First, note that necessarily $U^Tx\neq 0$, since otherwise we would have $I_p = \tilde{U}^T \tilde{U} = I_p +\|x\|^2 yy^T$
and thus $x=0$ or $y=0$. Since both $U$ and $\tilde{U}$ have orthonormal columns, it holds $\|U^T\tilde{U} \| \leq \|U\|\|\tilde{U}\| = 1$. The principal angles between the subspaces $[U]$ and $[\tilde{U}]$ are determined by the singular values of  $U^T\tilde{U}=I_p + U^Txy^T$,
which is a non-trivial rank-one modification of the $(p\times p)$ identity matrix.
By \cite[Theorem 1]{Thompson1976}, the singular values $\beta_1\geq\ldots\geq\beta_p$ of $B:=U^T\tilde{U}$ are sandwiched between the singular values of $I_p$ in the following way:
\[
 \alpha_{k-1} \geq \beta_k \geq \alpha_{k+1}, \quad k=1,\ldots,p,
\]
where in the case at hand, $\alpha_0 = \infty, \alpha_1=\cdots=\alpha_p=1, \alpha_{p+1} = 0$. 
Combined with the fact that $\beta_1 = \|U^T\tilde{U} \|$, this entails $\beta_k=1$, $k=1,\ldots, p-1$ and $1\geq \beta_p\geq 0$. In fact, $1>\beta_p$, for otherwise, $U^T\tilde{U}= B = I_p + U^Txy^T$ would be an orthogonal matrix.\footnote{As an aside, in this case $B$ would necessarily be a Householder reflection with $y = \frac{-2}{\|U^Tx\|^2}U^Tx$ so that $B = I-\frac{2}{\|U^Tx\|^2}U^Txx^TU$.}
But then $[U] = [\tilde{U}]$, $\tilde{U} = UB$.
Hence, $\tilde{U} - U = U(B-I)= xy^T$, which implies $x\in \colspan(U)$, a contradiction.
As a consequence, the principal subspace angles are
\[
 \theta_k = \arccos(\beta_k) = 0, \quad k=1,\ldots,p-1,\quad \mbox{and } \frac{\pi}{2}\geq \theta_p = \arccos(\beta_p) > 0. 
\]
\end{proof}

Note that the implicit requirement that $\tilde{U} = U + xy^T\in St(n,p)$, i.e., that the rank-one update by $xy^T$ preserves the mutual orthonormality of the columns of $\tilde{U}$, imposes special constraints on the selection of the vectors $x,y$.
Intuitively, we expect that any two subspaces $\colspan(X)$ and $\colspan(X+ab^T)$ 
differ by a rank-one modification in the sense of Definition \ref{def:ROmods}.
This can be deduced as an immediate consequence of Lemma \ref{lem:ROchar} and the upcoming Theorem \ref{thm:main}.
Yet, one can also establish this directly: For example, the QR-update procedure \eqref{eq:QR_rank1} gives 
$$X + ab^T = QR + ab^T = (Q,q) \tilde{K} = (Q,q) \begin{pmatrix}\phi\\ \varphi^T\end{pmatrix}\tilde{R} = (Q\phi + q\varphi^T)\tilde{R}.$$
Here, $\Phi =\begin{pmatrix}\phi\\ \varphi^T\end{pmatrix}\in St(p+1,p)$, so that 
$(Q\phi + q\varphi^T)\in St(n,p)$ but $Q\phi \not\in St(n,p)$ in general.
The principal angles are determined by the singular values of $Q^T(Q\phi + q\varphi^T) = \phi$.
The singular values of $\phi$ are the square roots of the eigenvalues of $\phi^T\phi = I_p - \varphi\varphi^T$.
By \cite[Theorem 1]{BunchNielsenSorensen1978}, all eigenvalues and thus all singular values are equal to $1$
with the smallest one as the only exception. Thus, we have proved
\begin{corollary}
 \label{cor:ROmods}
 Two subspaces $\colspan(X)$ and $\colspan(X+ab^T)$ with $a\not\in\colspan(X)$ differ by a non-trivial rank-one modification in the sense of
 Definition \ref{def:ROmods}. 
 In particular, there exist a subspace representative $U$ and vectors $x$, $y$ such that
 both $U\in St(n,p)$ and $U+xy^T\in St(n,p)$ and, in addition,
 $[U] = \colspan(X)$, $[U+xy^T] = \colspan(X+ab^T)$.
\end{corollary}

How to actually obtain such structure-preserving vectors $x,y$ from a given {\em arbitrary} rank-one update in closed form is an alternative way to state the main objective of this work.

We now turn to the geometric solution of the rank-one subspace adaptation problem of Section \ref{sec:probstate}.
The idea is to find a geodesic path on the Grassmann manifold $Gr(n,p)$ that connects $[U]$ and $[U_{new}]$.
As outlined in Section \ref{sec:NotAndPrelims}, such a geodesic is determined by a starting point $[U]$ and a starting velocity $\Delta \in T_{[U]}Gr(n,p)$ which results in the expression \eqref{eq:grassgeo}.
In the special case, where the tangent velocity $\Delta$ is a rank-one matrix $\Delta = d v^T$,
the compact SVD of $\Delta$ is 
\begin{equation}
\label{eq:rank1grad_svd}
\Delta = \Phi\begin{pmatrix}[c|c]
                                  s &   \\
                                  \hline
                                    & 0_{p-1} \\
                                 \end{pmatrix} \Psi^T
                                 = \frac{d}{\|d\|} (\|d\|\|v\|) \frac{v^T}{\|v\|} =:  q s w^T\in \R^{n\times p}
\end{equation}
and the formula for the geodesic \eqref{eq:grassgeo} becomes
\begin{equation}
\label{eq:rank1geo}
 t\mapsto Exp_{[U]}(t\Delta)=[U +\left((\cos(ts)-1) Uw + \sin(ts)q \right) w^T] =: [U + \hat{u}(t)w^T].
\end{equation}
This shows that following a geodesic path along a rank-one tangent direction corresponds to a matrix curve of
rank-one updates, where $U + \hat{u}(t)w^T\in St(n,p)$ for each $t$.
Conversely, this motivates the conjecture that a rank-one adaptation $[U_{new}]$ of a given subspace representative $[U]$
can be reached via a geodesic path along a rank-one tangent direction.
The obstacle is to find the associated tangent direction $\Delta$.\footnote{If $[U_{new}]$ and $[U]$ are both known, then $\Delta$ can be computed via the Riemannian logarithm, i.e., the inverse of the exponential. The difficulty here is that $[U_{new}]$ is precisely the quantity that is sought after.}

This is the main result of this work: Given a rank-one adaptation of a  matrix $X + ab^T$
with $\colspan(X) = [U]$, $U\in St(n,p)$, we find a rank-one tangent vector 
$\Delta\in T_{[U]}Gr(n,p)$ of unit norm
and a step $t^*\in\R$ such that the associated geodesic crosses the adapted subspace 
at $t^*$. More precisely,
\[
 \colspan(X + ab^T)\stackrel{!}{=} [\tilde{U}_{new}] = Exp_{[U]}(t^*\Delta),
\]
with an orthogonal subspace representative $\tilde{U}_{new}\in St(n,p)$,
see Fig. \ref{fig:manifold_plot_SVD_up}. 
Note that we can obtain a subspace representative $U\in St(n,p)$ with $\colspan(X) = [U]$
from an SVD or a QR-decomposition of $X$. Both alternatives boil down to
a decomposition of the form $X = UW$ with $U\in St(n,p)$ and $W\in \R^{p\times p}$ regular and the rank-one update on $X$
leads to a rank-one update on $U$ via
$X+ab^T = (U + ab^TW^{-1})W$.

%
%
An important building block is the following lemma, which is taken from 
\cite{ZimmermannPeherstorferWillcox_2017}. It addresses the modification
of the orthogonal projector $\Pi_{X_{new}}= U_{new}U^T_{new}$ under the rank-one update on $X$ in closed form. 
The original lemma addressed the SVD-case of $W= \Sigma V^T$.
Adjusted to the general setting of $X= UW$, it reads
%
%
%
%
\begin{lemma}[\cite{ZimmermannPeherstorferWillcox_2017}]
 \label{lem:rank_one_proj}
 Let $X\in\R^{n\times p}$ feature a decomposition of 
    $X =  UW$ with $U\in St(n,p)$ and $W\in \R^{p\times p}$ regular. 
    Let 
    $X_{new} = X + ab^T$ and
 define
 \begin{subequations}
  \begin{align}
   \tilde{q} &= (I - UU^T)a, \quad  q = \frac{\tilde{q}}{\|\tilde{q}\|_2} \in \R^{n}, \\
   g &= \begin{pmatrix}
        \tilde{w}\\
        \omega
       \end{pmatrix}
      = \begin{pmatrix}
        - W^{-T}b\\
        \frac{1}{\|\tilde{q}\|_2}(1+a^TU W^{-T}b)
       \end{pmatrix} \in \R^{p +1}.
  \end{align}
 \end{subequations}
 Then the orthogonal projection onto $\colspan(X_{new})$ is
 \begin{equation}
 \label{eq:rank_one_proj}
  \Pi_{X_{new}} = (U,q) \begin{pmatrix}
                         U^T\\
                         q^T
                        \end{pmatrix}
                -\frac{1}{\|g\|_2^2} (U,q)g g^T
                        \begin{pmatrix}
                         U^T\\
                         q^T
                        \end{pmatrix}.
 \end{equation}
\end{lemma}
For the sake of completeness, a proof of the lemma is included in the appendix.

We are now in a position to state the main theorem.
\begin{theorem}
\label{thm:main}
In the same setting as above, consider the rank-one update $X_{new} = X + ab^T = U W+ ab^T$.
Then, 
\begin{equation}
\label{eq:RO-update-final}
 U_{new} = U + \left(\alpha Uw + \beta q\right) w^T
\end{equation}
is a valid matrix subspace representative $U_{new} \in St(n,p)$ of the 
rank-one modified subspace such that $[U_{new}] =\colspan(X_{new}) \in Gr(n,p)$.
In particular, $\Delta := qw^T$ is a rank-one tangent vector $\Delta\in T_{[U]}Gr(n,p)$ and the geodesic that starts at $[U]$ with velocity $\Delta$ meets the point $\colspan(X_{new})$ on the Grassmann manifold.\\
The quantities that appear in \eqref{eq:RO-update-final} are defined as follows:
\begin{subequations}
  \begin{align}
  \nonumber
  \tilde{q} = (I-UU^T)a, \quad q = \frac{\tilde{q}}{\|\tilde{q}\|}, \quad &  \tilde{w} = -W^{-T}b, \quad w = \frac{\tilde{w}}{\|\tilde{w}\|},\\
    \nonumber
  \omega = \frac{1}{\|\tilde{q}\|}(1-a^TU\tilde{w})\in \R, \quad& g = (\tilde{w}, \omega)^T\in \R^{p+1},\\
    \nonumber
  \alpha = \frac{|\omega|}{\|g\|} - 1, \quad & \beta = -\sign(\omega) \frac{\|\tilde{w}\|}{\|g\|} \in \R.
  \end{align}
\end{subequations}
\end{theorem}
\begin{proof}
 First, note that by construction $\Delta^T U = w(q^TU) = 0$, so that indeed $\Delta \in T_{[U]}Gr(n,p)$, see \eqref{eq:GrassTangent}.
 Consider the corresponding geodesic
 \begin{equation}
\label{eq:RO-update-geo}
 t\mapsto [U](t) = Exp_{[U]}(t\Delta)=[U +\bigl((\cos(ts)-1) Uw + \sin(ts)q \bigr) w^T] 
 \end{equation}
 From Lemma \ref{lem:rank_one_proj}, we know that the orthogonal projection onto $\colspan(X_{new})$ is
 \begin{subequations}
  \nonumber
  \begin{align}
   \Pi_{X_{new}} &= 
   UU^T + qq^T - \frac{1}{\|g\|^2}\left(U\tilde{w} + \omega q \right)
   \left(\tilde{w}^TU + \omega q^T \right)\\
   &= UU^T + \left(1-\frac{\omega^2}{\|g\|^2}\right) qq^T
    - \frac{\|\tilde{w}\|^2}{\|g\|^2}\left(Uww^TU^T\right)\\
   & - \frac{\omega \|\tilde{w}\|}{\|g\|^2}\left(Uwq^T + qw^TU^T\right).
  \end{align}
 \end{subequations}
 The geodesic \eqref{eq:RO-update-geo} leads to a curve of orthogonal projectors
 $t\mapsto\Pi_{U(t)}= U(t)U(t)^T$. An elementary calculation shows that
  \begin{subequations}
  \nonumber
  \begin{align}
   \Pi_{U(t)}
   &= UU^T + \sin^2(t)qq^T
    - \sin^2(t)\left(Uww^TU^T\right)\\
   & + \cos(t)\sin(t)\left(Uwq^T + qw^TU^T\right).
  \end{align}
 \end{subequations}
 Comparing the expressions of $\Pi_{X_{new}}$ and $\Pi_{U(t)}$ term by term,
 we see that the task is reduced to find $t^*$ such that
   \begin{equation}
  \nonumber
   \left(1-\frac{\omega^2}{\|g\|^2}\right)
   = \sin^2(t^*) = \frac{\|\tilde{w}\|^2}{\|g\|^2}, \quad \mbox{and}\quad
   -\frac{\omega \|\tilde{w}\|}{\|g\|^2} = \cos(t^*)\sin(t^*).
 \end{equation}
 This is indeed possible:
 First, recall that $g= (\tilde{w}^T,\omega)^T$ and observe that
 \[
  \left(1-\frac{\omega^2}{\|g\|^2}\right)
  = \frac{\|g\|^2 -\omega^2}{\|g\|^2} = \frac{\|\tilde{w}\|^2}{\|g\|^2}<1.
 \]
 Hence, $\sin(t^*) = \pm \frac{\|\tilde{w}\|}{\|g\|}$.
 
 As a consequence,
 \begin{eqnarray*}
  \cos(t^*)\sin(t^*) &=& \sqrt{1-\sin^2(t^*)} \sin(t^*)
  =\sqrt{\frac{\|g\|^2 - \|\tilde{w}\|^2}{\|g\|^2}} (\pm 1)\frac{\|\tilde{w}\|}{\|g\|}\\
  &=&\pm \frac{|\omega|}{\|g\|} \frac{\|\tilde{w}\|}{\|g\|}
  \stackrel{!}{=} - \frac{\omega}{\|g\|} \frac{\|\tilde{w}\|}{\|g\|}.
 \end{eqnarray*}
 Hence, the sign must be chosen such that $\sin(t^*) = -\sign(\omega) \frac{\|\tilde{w}\|}{\|g\|}$
 and we obtain
 \[
  \alpha = \cos(t^*) -1 = \frac{|\omega|}{\|g\|}-1,\quad \beta = \sin(t^*) = -\sign(\omega) \frac{\|\tilde{w}\|}{\|g\|}.
 \]
 The corresponding step is $t^* = \arcsin\left(-\sign(\omega) \frac{\|\tilde{w}\|}{\|g\|}\right)$.
 For this choice of $t^*$, the orthogonal projectors $\Pi_{X_{new}}=\Pi_{U(t^*)}$ coincide.
 Due to the one-to-one correspondence between a subspace and the orthogonal projection onto it, we obtain
 $$[U](t^*) = \colspan(X_{new}),$$
 as claimed.
\end{proof}
Computing the orthogonal component $\tilde{q} = (I-UU^T)a$ consists of a single classical Gram-Schmidt step.
For stability reasons, this may replaced by a modified Gram-Schmidt step.
%
%
%
%
\subsection*{Subspace distance}
The theorem allows us to compute the Riemannian subspace distance between the original subspace 
$\colspan(X) = [U]$ and the adapted subspace $\colspan(X + ab^T) =  [U_{new}]$ immediately from the given data.
\begin{corollary}
 \label{cor:subspace_dist}
 The Riemannian subspace distance between the original subspace 
$\colspan(X) = [U]$ and the adapted subspace $\colspan(X + ab^T) =  [U_{new}]$
is
\[
 \mbox{dist}([U], [U_{new}]) = \arccos\left(\frac{|\omega|}{\|g\|}\right)
 = \arccos\left(\frac{|\omega|}{\sqrt{\|\tilde{w}\| + \omega^2}}\right),
\]
where $\tilde{w}, \omega$ are as introduced in Theorem \ref{thm:main}.
\end{corollary}
\begin{proof}
Let $w$ be as in the theorem and let $W^\bot$ be an orthogonal completion such that $Z=(W^\bot, w)\in O(p)$.
From \eqref{eq:RO-update-final}, we have
\begin{subequations}
 \nonumber
 \begin{align}
  U^TU_{new} &= U^T\left(U + \left(\alpha Uw + \beta q\right) w^T\right) = I + \alpha ww^T\\
   & = I + \alpha Z \begin{pmatrix}[c|c]
                                  0_{p-1} &   \\
                                  \hline
                                    & 1
                                 \end{pmatrix} Z^T\\
&     =Z\left(
      \begin{pmatrix}[c|c]
      I_{p-1} &   \\
      \hline
      & 1
     \end{pmatrix}
     + 
     \begin{pmatrix}[c|c]
      0_{p-1} &   \\
      \hline
      &  \cos(t^*) -1
     \end{pmatrix}\right) Z^T.
 \end{align}
\end{subequations}
Hence, the subspace distance is 
\[
 \|(\arccos(1),\ldots,\arccos(1), \arccos(\cos(t^*)))^T\| = |t^*| = \arccos\left(\frac{|\omega|}{\|g\|}\right).
\]
(This can also be seen by converting \eqref{eq:RO-update-final} back to the general form \eqref{eq:grassgeo}.)
\end{proof}
\subsection*{Recovering $\mathbf{X_{new}}$}
If $X\in \R^{n\times p}$ has a decomposition $X=UW$ with $U\in St(n,p)$, then Theorem \ref{thm:main}
gives $U_{new}\in St(n,p)$ such that $\colspan(X_{new}) = \colspan(X+ab^T) = [U_{new}]$.
We may use this to construct a decomposition $X_{new}=U_{new}W_{new}$.
In fact, since $\colspan(X+ab^T) = [U_{new}]$, it holds
\[
 X_{new} = X+ab^T = UW + ab^T = U_{new}W_{new},
\]
with a suitable $W_{new}\in \R^{p\times p}$. Multiplying with $U_{new}^T$ from the left gives
\[
 W_{new} = U_{new}^TUW + U_{new}^Tab^T.
\]
This is also clear from the fact that $U_{new}U_{new}^TX_{new} = X_{new}$.
By inserting the explicit formula for $U_{new} = U + \left(\alpha Uw + \beta q\right) w^T$,
the updated $W_{new}$ is obtained from a rank-one update on $W$ via
\begin{align}
\label{eq:W_update}
 W_{new} &= W + \left(U^Ta + \gamma w\right) b^T,
 & \gamma = \left(\beta (q^Ta) - \alpha \frac{\|\tilde{q}\|\omega}{\|\tilde{w}\|}\right) \in \R,
\end{align}
where all quantities are as introduced in Theorem \ref{thm:main}.
As a consequence, the rank-one update $X_{new} = UW + ab^T = U_{new}W_{new}$
splits into a orthogonality-preserving rank-one update on $U$, which gives $U_{new}$ and an associated 
rank-one update on $W$, which makes the geometric update of the orthogonal decomposition $X=UW$
very efficient. When compared to the classical SVD- or QR updates, one looses the special structure of the $W$-factor, though.
Yet, this may be reestablished by exclusively operating on $p$-by-$p$-matrices, when storing the leftmost subspace representative $U$ and the factors of the $W$-decomposition separately, as suggested in \cite{Brand2006}.
\subsection*{Computational complexity}
Computationally, Theorem \ref{thm:main} reduces the rank-one modified orthogonal decomposition to an elementary matrix update
\[
 U_{new} = U + xy^T \in \R^{n\times p}.
\]
Computing $x = U(\alpha w) + \beta q$ requires the following operations:
\begin{subequations}
  \begin{align}
  \nonumber
  \bar{a} = U^Ta: \hspace{0.2cm} np \mbox{ FLOPS}, \hspace{0.2cm} &  \tilde{q} = a-U\bar{a}:\hspace{0.2cm} np +n \mbox{ FLOPS}\\
    \nonumber
  q = \frac{\tilde{q}}{\|\tilde{q}\|}: \hspace{0.2cm} 2n \mbox{ FLOPS}, \hspace{0.2cm} &  \tilde{w} = -W^{-T}b, \hspace{0.2cm} w = \frac{\tilde{w}}{\|\tilde{w}\|}:\hspace{0.2cm} \mathcal{O}(p^3)\mbox{ FLOPS}\\
    \nonumber
  \omega = \frac{1}{\|\tilde{q}\|}(1-\bar{a}^T\tilde{w}): \hspace{0.2cm} \mathcal{O}(p) \mbox{ FLOPS} \hspace{0.2cm}& g = (\tilde{w}^T, \omega)^T, \|g\|:\hspace{0.2cm} \mathcal{O}(p)\mbox{ FLOPS} \\
    \nonumber
  \alpha = \frac{|\omega|}{\|g\|} - 1:\hspace{0.2cm} \mathcal{O}(1)\mbox{ FLOPS} , \hspace{0.2cm} & \beta = -\sign(\omega) \frac{\|\tilde{w}\|}{\|g\|}:\hspace{0.2cm} \mathcal{O}(1)\mbox{ FLOPS}.
  \end{align}
\end{subequations}
Assuming that $n\gg p$, we count only the operations that scale in $n$:
These sum up to $3np + 4n$ (computing $x$) + $np$ (computing $U+xy^T$) = $4np + 4n= \mathcal{O}(np)$ FLOPS.

Note that all the terms that appear in the $W$-update \eqref{eq:W_update}
are already available from the computations for $U_{new}$.
Therefore, after $U_{new}$ is known, the corresponding $W_{new}$ is obtained via an elementary rank-one update
on the $(p\times p)$-matrix $W$ which consumes $\mathcal{O}(p^2)\subset \mathcal{O}(np)$ FLOPS,
see \eqref{eq:W_update}. 
Hence, both factors $U_{new}$, $W_{new}$ of the complete update of the orthogonal decomposition
\[
 X_{new} = UW + ab^T = U_{new}W_{new}.
\]
are obtained in $\mathcal{O}(np)$ FLOPS.\footnote{The Level-2 BLAS operation ``DGER'' performs the rank-one operation $A := \alpha xy^T + A$.}

In contrast, the rank-one adaptations \eqref{eq:brandupdate1} and \eqref{eq:QR_rank1} require the matrix-matrix product
of an $n\times (p+1)$ matrix with a $(p+1)\times p$-matrix for computing the adapted subspace representative, which alone takes $n(p+1)p = \mathcal{O}(np^2)$ FLOPS.
%
\section{Relation to existing work}
\label{sec:related}
\paragraph{The Wedderburn rank reduction theorem}
The Wedderburn-Egerv{\'a}ry  rank reduction formula characterizes the rank-one modifications that reduce the rank of a given rectangular matrix $X\in\R^{n\times p}$ exactly by one, see \cite{ChuFunderlicGolub1995}. An account of the history of this theory is given in \cite{Galantai2010}. 
The precise statement is as follows:
If $x\in\R^p$, $y\in\R^n$ satisfy $0\neq y^TXx$, then
\begin{equation}
 \label{eq:wedderburn}
 X_{new} := X - \frac{1}{y^TXx} (Xx) (y^TX)
\end{equation}
has rank exactly one less than the unmodified $X$.
Conversely, if $a\in\R^n$, $b\in\R^p$ and $\rho\in\R\setminus\{0\}$
are such that $X_{new} = X- \frac{1}{\rho} ab^T$ has a rank exactly one less than $X$, then
there exist $x\in\R^p$, $y\in \R^n$ such that $a=Xx$ and $b=X^Ty$ and $\rho = y^TXx$.

In particular, the rank of $X_{new}= X - \frac{1}{\rho} ab^T$ can only decrease if the vector $a\in\R^n$ is {\em in the range of $X$.}
Exactly this case is excluded in the considerations of this work.
\paragraph{Grassmannian Rank-One Update Subspace Estimation}
The Grassmannian Rank-One Update Subspace Estimation (GROUSE, \cite{balzano2010online})
 considers the unique, subspace-dependent residual associated with a least-squares problem of the form
  \begin{equation*}
    \argmin_{b \in \R^p} \|A^T Ub - a\|_2^2
  \end{equation*}
as a differentiable function on the Grassmannian
 \begin{equation*}
  F: Gr(n, d)  \rightarrow \R, \quad [U] \mapsto a^T(I - A^T U (U^TAA^T U)^{-1} U^TA) a.
\end{equation*}
GROUSE is an iterative optimization scheme that operates on a sequence of incoming, possible incomplete data vectors $a\in\R^n\setminus\{0\}$.
Each iteration is based on following the Grassmann geodesic in the direction of steepest descent $-\nabla_{[U]}F$ with respect to the above residual norm function $F$.
The Grassmann gradient is $\nabla_{[U]}F = -2A(a-A^TUb)b^T = -2Ar(b)b^T \in T_{[U]}Gr(n,p)$,
where $b = (U^TAA^TU)^{-1}U^TAa$ is the optimal coefficient vector associated with the above least-squares problem
and $a-A^TUb=r(b)$ is the corresponding residual vector.
Note that the gradient is of rank-one. GROUSE thus makes extensive use of the formulas 
\eqref{eq:rank1grad_svd}, \eqref{eq:rank1geo} for the specific rank-one descent directions $\Delta = qsw^T = 2Ar(b)b^T$.
In the simplest case, where $A=I_n$, it holds 
$q= \frac{(I-UU^T)a}{\|(I-UU^T)a\|}, w = \frac{b}{\|b\|}=\frac{U^Ta}{\|U^Ta\|}$, which shows that the left-singular vector $q$ and the right-singular vector $w$ are not independent quantities, since both vectors are functions of $a$ and $U$.
In any case, the singular vectors $q$ and $w$ from the SVD of the rank-one gradient $\nabla_{[U]}F$ are special in the sense that they correspond to a certain subset of rank-one tangent vectors that arise from least-squares problems, 
where both exhibit a functional dependency on $a, U$ (and $A$). In contrast, Theorem \ref{thm:main} considers completely general rank-one tangent vectors.
%
%
\appendix
\section{Proof of Lemma }
\begin{proof}
We start with a decomposition inspired by \cite[eq. (3)]{Brand2006}.
Note that $(U,q)\in St(n, p+1)$ by construction. It holds that
\begin{equation}
 \nonumber
 X + ab^T = (U,q)
    \begin{pmatrix}
      W + U^Tab^T\\
      \|\tilde{q}\|b^T
    \end{pmatrix}
   =: (U,q) K,
\end{equation}
where $K\in \R^{(p+1)\times p}$.
Let $K = \tilde{U}\tilde{W}$
be a decomposition of $K$ with $\tilde{U} \in St(p+1, p)$, $\tilde{W} \in \R^{p \times p}$ regular.
Hence,
\begin{equation}
 \nonumber
 X + ab^T = \left((U,q)\tilde{U}\right)\tilde{W} =: U_{new}W_{new}.
\end{equation}
Let $g\in \R^{p+1}$ be such that $(\tilde{U}, \frac{g}{\|g\|}) \in O_{p+1}$
is an orthogonal completion of $\tilde{U}$.
Because of
$
  I_{p+1} = (\tilde{U}, \frac{g}{\|g\|})(\tilde{U}, \frac{g}{\|g\|})^T,
$
we have
\[
  \tilde{U}\tilde{U}^T = I_{p+1} - \frac{1}{\|g\|^2}gg^T
\]
and, as a consequence,
\begin{equation}
\label{eq:aux_rank-one_proj}
 U_{new}U_{new}^T
 = (U,q)\tilde{U}\tilde{U}^T \begin{pmatrix}
                         U^T\\
                         q^T
                        \end{pmatrix}
 = (U,q)\left(I_{p+1} - \frac{1}{\|g\|^2}gg^T\right)
                        \begin{pmatrix}
                         U^T\\
                         q^T
                        \end{pmatrix}.
\end{equation}
Hence, it is sufficient to determine $g$, which is characterized up
to a scalar factor by $\tilde{U}^Tg =0$.
Since $\colspan(K) = \colspan(\tilde{U})$, this condition is equivalent to
$K^Tg=0$.
Let $\tilde{w}\in\R^{p}$ denote the first $p$ components of $g$
and let $\omega\in \R$ be the last entry such that
$g^T = (\tilde{w}^T,\omega)$.
When writing the equation $g^TK=0$ as
\[
 (\tilde{w}^T,\omega)\begin{pmatrix}
                 I_p & U^Ta\\
                 0      & \|\tilde{q}\|_2
                \end{pmatrix}
                \begin{pmatrix}
                 W\\
                 b^T
                \end{pmatrix} =0,
\]
it is straightforward to show that
$g  = \begin{pmatrix}
        -W^{-T}b\\
        \frac{1}{\|\tilde{q}\|_2}(1+a^TUW^{-T}b)
       \end{pmatrix} \in \R^{p +1}$
and any scalar multiple of this vector
is a valid solution.
Using this vector in \eqref{eq:aux_rank-one_proj} proves the lemma.
\end{proof}

\section{MATLAB code}
\paragraph{MATLAB function for performing the adaptation of Theorem \ref{thm:main}}
%
%
%
\begin{verbatim}
function [ Unew, Wnew, Sdist ] = grood(U, W, a, b)
%----------------------------------------------------------
% Perform a 
%    (G)eometric (R)ank-(O)ne (O)rthogonal (D)ecomposition
%     Unew Wnew = UW + ab'
% on a matrix X with known orthogonal decomposition X=UW
%
% Inputs:
%   UW: decomposition of X=UW into a column-orthogonal 
%       (nxp)-matrix U and a regular (pxp)-matrix W
%    a: n-vector
%    b: p-vector
%
% Outputs:
% Unew: column-orthogonal matrix with range 
%       ran(Unew) = ran(U+ab')
% Wnew: second factor in orthogonal matrix decomposition
%Sdist: subspace distance between ran(Unew) and ran(U)
%
% author: R: Zimmermann, IMADA, SDU Odense
%   zimmermann@imada.sdu.dk
%----------------------------------------------------------
% compute orthogonal component of a and normalize
Ua= U'*a;
q = a - U*Ua;
n_q = norm(q);
q_n  = q/n_q;
%
% proceed only, if a is not in ran(U)
if n_q > 1.0e-12
    w = linsolve(W', -b);
    % Grassmann update:
    omega = (1.0/n_q)*(1 - Ua'*w);
    n_w = norm(w);
    w_n = w/n_w;
    % compute the norm of g = (w, omega)
    n_g = sqrt(n_w^2 + omega^2);
    % compute the (cos(t)-1) and sin(t)-factors:
    sin_factor = -sign(omega)*(n_w/n_g);
    cos_factor = abs(omega)/n_g - 1;
    % compute the rank-one update
    rank1_up = (cos_factor*U*w_n + sin_factor*q_n)*w_n';
    Unew = U + rank1_up;
    % compute the update on W
    Wnew = W + (Ua + (sin_factor*(q_n'*a) ...
           - cos_factor*(n_q*omega)/n_w)*w_n)*b';
    %subspace distance
    Sdist = acos(abs(omega)/n_g);
else
    Unew = U;
    Wnew = W + Ua*b;
    Sdist = 0.0;
return;
end
\end{verbatim}

\end{document}